\documentclass[12pt]{amsart}
\usepackage{latexsym,amssymb}
\newtheorem{thm}{Theorem}[section]
\newtheorem{cor}[thm]{Corollary}
\newtheorem{lem}[thm]{Lemma}

\theoremstyle{remark}
\newtheorem{rem}[thm]{Remark}
\newtheorem{prop}[thm]{Proposition}
\theoremstyle{definition}
\newtheorem{defi}[thm]{Definition}

\newtheorem{claim}{Claim}

\newcommand{\norm}[2][]{\| \, {#2} \,\|_{#1}}
\newcommand{\abs}[1]{|{#1}|}

\newcommand{\NN}{\ensuremath{\mathbb{N}}}
\newcommand{\CC}{\ensuremath{\mathbb{C}}}
\newcommand{\RR}{\ensuremath{\mathbb{R}}}

\newcommand{\ZZ}{\ensuremath{\mathbb{Z}}}

\newcommand{\cL}{\mathcal{L}}
\newcommand{\spr}[2][]{sp_{#1}(#2)}
\newcommand{\tensor}{\otimes}
\newcommand{\projtensor}{\hat{\otimes}}
\newcommand{\injtensor}{\otimes_{\epsilon}}
\newcommand{\conv}{\star}
\newcommand{\card}[1]{\mbox{card}({#1})}
\newcommand{\spec}[2][]{\sigma_{#1}{(#2)}}
\newcommand{\sprad}[2][]{r_{#1}(#2)}
\newcommand{\essup}{\mbox{\rm ess sup}}
\newcommand{\Ruc}{Ruc}
\newcommand{\Aut}{Aut}
\newcommand{\END}{End}
\newcommand{\AF}{\ensuremath{E}}
\newcommand{\noteq}{\not=}

\begin{document}


\title[Convolution Dominated Operators]{On Convolution Dominated Operators}

\author{Gero Fendler}
\address{Fakult\"at f\"ur Mathematik, Universit\"at Wien,\\
Oskar-Morgenstern-Platz 1, A-1090 Wien, Austria}
\email{gero.fendler@univie.ac.at}

\author{Michael Leinert}
\address{Institut f\"ur Angewandte Mathematik, Universit\"at Heidelberg,\\
Im Neuenheimer Feld 294, 69120 Heidelberg, Germany}
\email{leinert@math.uni-heidelberg.de}
\date{\today}

\begin{abstract}
{For a locally compact group $G$ we consider the algebra $CD(G)$
 of convolution dominated operators on
$L^{2}(G)$:
An operator $A:L^2(G)\to L^2(G)$ is called {\em convolution dominated}
if there exists $a\in L^1(G)$ such that for all $f \in L^2(G)$
\begin{equation*}
\abs{Af(x)} \leq  a \conv \abs{f}\,(x), \; \mbox{for almost all } x\in G.
\end{equation*}
In the case of discrete groups those operators can be dealt with quite
sufficiently if the group in question is rigidly symmetric.
For non-discrete groups we investigate the
subalgebra of
{\em regular convolution dominated operators} $CD_{reg}(G)$.
\par
For amenable $G$ which is rigidly
symmetric as a discrete group, we  show that
any element of $CD_{reg}(G)$ is
invertible in $CD_{reg}(G)$
if it is invertible as a bounded operator on $L^2(G)$.
\par
We give an example of a symmetric group $E$ for which
the convolution dominated operators are not inverse-closed
in the bounded operators on $L^2(E)$.
}
\end{abstract} 

\maketitle

\section{Introduction}
When one considers a convolution operator on the abelian group $\ZZ$,
then its matrix with respect 
to the canonical basis of $l^2(\ZZ)$
is a Toeplitz matrix, i.e., it is constant along side diagonals. 
Conversely,  
a doubly infinite matrix,
which is constant along its
side diagonals and satisfies certain off-diagonal decay  conditions,
defines a convolution operator, when it is considered as acting with 
respect to the above basis.
For the class of 
operators with decay in the sense of $l^1$-summability
N.~Wiener \cite{Wiener32} proved the Fourier transformed version of the following
theorem
\begin{thm} [Wiener's Lemma]
If a two-sided infinite absolutely summable sequence
$a=\left( a(n)\right)_{n\in \ZZ}$ 
is invertible as a convolution operator on $l^2(\ZZ)$, then the inverse
is given by convolution with some $b\in l^1(\ZZ)$.
\end{thm}
Based on work of de Leeuw~\cite{deL} and of Bochner and Phillips~\cite{BP} 
on operator-valued Fourier transforms,
several authors extended and applied
Wiener's lemma to the case of matrices of operators,
for which each side diagonal
is uniformly bounded and these bounds are summable.
The index set always  had 
to be a discrete abelian group or a countable set~\cite{GoKaWo89a},  \cite{Sjo95},
\cite{GL01},  \cite {Groe06}, \cite{Sun07a}, \cite{Balan08}.
\par
If $G$ is a locally compact abelian group,
for simplicity assumed to be compactly generated,
then it admits a discrete co-compact subgroup $H$, and $L^p(G)=l^p(H,L^p(D))$,
where $D$ is some fundamental domain. So vector valued Wiener Lemmata become
applicable to classes of integral operators.
Baskakov in \cite{Bask90,Bask92,Bask97,Bask97a} derives results 
for some of these.
Shin and Sun in \cite{ShinSun13} give an account of those techniques.
\par
If $\mathcal{B}$ is a Banach algebra and $\mathcal{A}$ a subalgebra
then $\mathcal{A}$ is called spectral invariant in $\mathcal{B}$
when every element of $\mathcal{A}$ has the same spectrum in  $\mathcal{A}$
as  it has in $\mathcal{B}$. In \cite{Kurb99} and \cite{Kurb01} Kurbatov 
shows among other results that for a locally compact abelian group
the algebra of convolution dominated operators 
(see \ref{def:convdom}) is spectral in the bounded operators on $L^p(G)$.
Farrell and Strohmer \cite{FarrStro10} extend this result to generalised 
Heisenberg groups (with compact center). 
\par
Plenty of work has been done with regularity assumptions on the kernel
of an integral operator. Let us just mention  interesting studies
by Gr{\"o}chenig and Klotz \cite{GroechKlotz10,GrKl13,GrKl14,Klotz12,Klotz14} 
on norm controlled inverse closedness of 
smoothness algebras in symmetric algebras. 
A comprehensive discussion of this important theme
would go beyond the scope of this note.
\par
In \cite{fgl08}, together with Gr{\"o}chenig we
addressed the discrete nonabelian case using tools 
from abstract harmonic analysis to
circumvent the restrictions of abelian Fourier transformation: 
\begin{thm}\label{theorem:discrete}
Let $G$ be a discrete finitely generated group of polynomial
growth. If a matrix $A$ indexed by $G$ satisfies the off-diagonal
decay condition 
\begin{equation}\label{cd:dis}
|A(x,y)| \leq a(xy^{-1}) , x,y\in G
\end{equation} 
for some $a\in \ell ^1(G)$ and $A$ is
invertible on $\ell ^2(G)$, then there exists $b\in \ell ^1(G)$  such
that 
\[|A^{-1}(x,y)| \leq b(xy^{-1}), x,y\in G,\] 
i.\ e.\ the algebra of matrices satisfying (\ref{cd:dis})
is inverse-closed in $B(L^2(G))$, the bounded operators on
$L^2(G)$.
\end{thm}
Throughout this paper it is assumed that the involution of a 
Banach  {$\ast$}-algebra is isometric.
We recall the main line of the proof of this theorem, since
this will be a guideline for us.
\begin{itemize}
\item[(i)]
The matrices satisfying (\ref{cd:dis}) form a
Banach {$\ast$}-algebra. We denote it by $CD(G)$.
\item[(ii)] Identify $CD(G)$ with  $l^1(G,l^{\infty}(G),T)$,
a twisted $l^1$-algebra in the sense of Leptin~\cite{lep67}, \cite{lep67a}.
\item[(iii)]
Based essentially on work of Leptin and Poguntke,
 in particular~\cite{LP79}, we proved 
that $CD(G)$ is symmetric Banach $\ast$-algebra.
\item[(iv)] In a final step one relates the symmetry of a Banach $\ast$-algebra
to the invertibility of certain of its elements. 
This is done with the help of Hulanicki's Lemma~\cite{hul72}.
\end{itemize}
\begin{defi}
\begin{itemize}
\item[(i)]A Banach {$\ast$}-algebra $A$ is
  called symmetric if for all $a\in A$ 
\[\spec[A]{a^* a}\subset [0,\infty),\]
where $\spec[A]{a^* a}$ denotes the spectrum of $a^{\ast}a$ in $A$.
\item[(ii)]Accordingly, a locally compact group
  $G$ is called symmetric,
if $L^1(G)$ is a symmetric Banach $\ast$-algebra.
\item[(iii)]A locally compact (l.c.) group $G$ is called rigidly symmetric 
if for any $C^{\ast}$-algebra $A$ the Banach $\ast$-algebra
$L^1(G)\projtensor A$
is symmetric, where the projective tensor product
$L^1(G)\projtensor A$ is endowed with its natural
$\ast$-algebra structure.
\end{itemize}
\end{defi}

In these notes we derive variants of  Theorem \ref{theorem:discrete}
for integral operators on non-discrete locally compact groups.
This extension follows the ideas of the discrete case in~\cite{fgl08},
but requires non-trivial modifications. The special problem to be addressed is the measurability and integrability of certain kernels.
We extend the results presented in the exposition \cite{FGL10} and provide  proofs for them. 
Meanwhile some of our results have been reproved by I.~Belti{\c{t}}{\u{a}} and D.~Belti{\c{t}}{\u{a}} \cite{BelBel15,BelBel15a} and by M\u{a}ntoiu \cite{Man15} in the discrete case. After archiving (on arxiv) this paper
A.~R.~Schep kindly informed us that our proposition \ref{prop:kernel} 
is a special case
of a theorem in his dissertation\cite{Schep77}.
\par
It is not known yet if a symmetric group is already rigidly symmetric.
An interesting example is the real $ax+b$ group $\AF$,
which is known to be a symmetric group but not symmetric as a discrete group~\cite{jen69}.
We shall discuss a specific example of a non-symmetric algebra related
to $\AF$ showing non-symmetry of $CD_{reg}(\AF)$.
So, in particular, $CD_{reg}(\AF)$ is not inverse-closed in $B(L^2(\AF))$.

\section{Convolution Dominated Operators}

Let $G$ be a locally compact group.
\begin{defi}\label{def:convdom}
A bounded operator $A:L^2(G)\to L^2(G)$ is called a convolution dominated operator if there exists $a\in L^1(G)$ such that for all $f \in L^2(G)$
\begin{eqnarray}\label{def:cdop}
\abs{Af(x)}& \leq & a \conv \abs{f}(x), \quad \mbox{for almost all } x\in G.
\end{eqnarray}
We define a norm on the space of convolution dominated operators by
\[\label{def:cdnorm}
\norm[CD]{A} = \inf\{\, \norm[1]{a} : (\ref{def:cdop}) \mbox{ holds true }\}
\]
and denote by $CD$ the normed linear space of convolution dominated operators.
\end{defi}
It is clear that in this definition, necessarily,
$a\geq 0$ locally almost everywhere and
that the $CD$-norm dominates the operator norm, in fact
for $1\leq p \leq\infty$ and $a\in L^1(G)$ satisfying (\ref{def:cdop})
\begin{eqnarray*}
 \norm[p]{Af}&=&(\int_G \abs{Af}^p(x) dx)^{1/p}\;\leq\; %
(\int_G (a \conv \abs{f})^p dx)^{1/p}\\
&\leq & \norm{\lambda(a)}\norm[p]{f}%
\;\leq\; \norm[1]{a}\norm[p]{f}.
\end{eqnarray*}
It follows that $A\in CD$ extends to a bounded
operator on $L^p(G),\; 1\leq p <\infty$ by continuity, and by duality to a bounded
operator on $L^{\infty}(G)$ too.
\begin{prop}
With the involution of operators on $L^2(G)$ and composition
of operators as product the space of convolution dominated operators is
a Banach $\ast$-algebra.
\end{prop}
\begin{proof}
Let $A,B$ be convolution dominated operators and choose 
$a, \mbox{ resp. }b \in L^1(G)$ according to (\ref{def:cdop}).
Then for $f\in L^2(G)$:
\[
\abs{A\circ B f}(x) \leq a\conv \abs{B f}(x) \leq a\conv b\conv \abs{f}(x).
\]
From this and because $L^1(G)$ is an normed algebra under convolution
it is clear that
\[\norm[CD]{A\circ B}\leq \norm[CD]{A} \norm[CD]{B}.\]
To see that the involution preserves the space $CD$ take $a\in L^1(G)$
according to (\ref{def:cdop}) and
$f,h\in L^2(G)$:
\begin{eqnarray*}
\abs{\int_{G} A^{\ast}{h}(x)\,\overline{f(x)}\,dx}&=&\abs{ ( A^{\ast}h,f)}\\
&=&\abs{ ( h,Af)}\\
&=&\abs{ \int_{G} h(x) \overline{Af(x)}\,dx}\\
&\leq &\int_{G} \abs{h(x)}\abs{Af(x)}\,dx\\
&\leq & \int_{G} \abs{h(x)} \, a \conv \abs{f}(x)\,dx\\
&=& \int_{G} a^{\ast}\conv \abs{h}(x)\,\abs{f}(x)\,dx,
\end{eqnarray*}
where $a^{\ast}(x)=\overline{a(x^{-1})}\Delta(x^{-1})$ 
is the involution on $L^1(G)$, and here, since $a\geq 0$, 
$a^{\ast}(x)={a(x^{-1})}\Delta(x^{-1})$.
Hence, $\abs{A^{\ast}{h}(x)} \leq a^{\ast}\conv \abs{h}(x)$ 
locally almost everywhere for all $h\in L^2(G)$, and it follows that 
$A^{\ast}\in CD$ with 
$\norm[CD]{A^{\ast}} \leq %
\inf \{ \norm[1]{a} : (\ref{def:cdop}) \mbox{ holds true }\}=\norm[CD]{A}$.
\par
To show that $CD$ is a complete space we let $(A_i)_{i\in \NN}$
be a sequence in $CD$ with $\sum_{i\in \NN} \norm[CD]{A_i}$ convergent.
Then we find $a_i\in L^1(G)$
such that $\abs{A_{i}f(x)} \leq  a_{i} \conv \abs{f}(x),%
 \quad \mbox{for almost all } x\in G$ and 
$\norm[1]{a_i}\leq\norm[CD]{A_i} + 2^{-i}$.
For $A=\sum_{i\in \NN} A_i$ this sum is convergent in the space of bounded operators, hence for $f\in L^2(G)$ and for a subsequence of the partial sums
we have 
$Af(x)=\lim_k\sum_{i=1}^{j_k} \! A_if(x)$ almost everywhere.
So $ \abs{Af}(x)\leq \sum_{i\in \NN} \!  \abs{A_i f}(x)%
\leq \sum_{i\in \NN} a_i \conv \abs{f} (x) = a \conv\abs{f} (x)$.
This shows $A\in CD$ with $\norm[CD]{A}\leq \norm[1]{a}$.
\end{proof} 
\begin{prop}\label{prop:kernel}
For a given convolution dominated operator $A$ there
exists a locally integrable function
$F_A:G \times G \to \CC$ such that 
for all 
$f\in C_{cp}(G)$
\[
Af(x)=\int_G F_A (x,y) f(y) dy,\quad \mbox{almost everywhere.}
\]
\end{prop}
\begin{proof}
let $K,K'\subset G$ be compact sets and consider on $C(K)\times C(K')$
the form
\[<Af,h> := \int_G Af(x)h(x) dx,\]
where $f$ and $h$ are extended as $L^2$-functions vanishing outside
$K$ respectively $K'$.
Take $\varphi\in C_{cp}(G)$, $\varphi \geq 0$,
then for $f_1, \ldots ,f_n \in C(K)$ and $h_1, \ldots , h_n \in C(K')$
\begin{eqnarray*}
\abs{\sum_{i=1}^n<\lambda(\varphi)\circ Af_i,h_i>} & = & %
\abs{\int_G \sum_{i=1}^n Af_i(x)(\varphi^{\ast}\conv h_i)(x)\,dx}\\
& = &\abs{\int_G\int_G %
A(\sum_{i=1}^n f_i h_i(z^{-1}))(x)\varphi^{\ast}(xz)\, dz dx}\\
& \leq & \int_G \int_G %
(a \conv \abs{\sum_{i=1}^n f_i h_i(z^{-1})}) (x) \varphi^{\ast}(xz)\, dzdx\\
& = & \int_{G} \int_{G} \int_{G} a(xy)\\
&&\quad\abs{\sum_{i=1}^nf_i(y^{-1}) h_i(z^{-1})} \varphi^{\ast}(xz)\,dy dzdx\\
& = & \int_{G} \int_{G} \int_{G} a(xy^{-1})\Delta(y^{-1})\cdot\\
& &\quad\abs{\sum_{i=1}^nf_i(y) h_i(z)} \varphi^{\ast}(xz^{-1})\Delta(z^{-1})%
\,dy dzdx\\
\end{eqnarray*}
\begin{eqnarray*}
\quad\quad\quad& \leq & \int_{G}%
\sup_{(y,z)\in K\times K'}\abs{\sum_{i=1}^nf_i(y) h_i(z)}\cdot\\
& &\quad\int_{K'}\int_{K}a(xy^{-1})\Delta(y^{-1})%
\varphi^{\ast}(xz^{-1})\Delta(z^{-1})%
\,dy dzdx\\
& \leq & \norm[1]{a}\sup_{(y,z)\in K\times K'}\abs{\sum_{i=1}^nf_i(y) h_i(z)}\cdot\\
& &\quad%
\int_{G}\int_{K'}\varphi^{\ast}(xz^{-1})\Delta(z^{-1})\,dzdx\\%
& = &\norm[1]{a}\sup_{(y,z)\in K\times K'}\abs{\sum_{i=1}^nf_i(y) h_i(z)}%
\abs{K'}\int_{G}\varphi(x)\,dx
\end{eqnarray*}
where $\Delta$ denotes the modular function of $G$,
$\varphi^{\ast}(z)=\varphi(z^{-1})\Delta(z^{-1})$ and we 
write $\abs{K'}$ for the left Haar measure of the set $K'$.
We take an approximate unit of $L^1(G)$ consisting of functions
$\varphi$ as above. Then we see 
that
\[
\abs{\sum_{i=1}^n<Af_i,h_i>}  \leq %
\abs{K'}\norm[1]{a}\sup_{(y,z)\in K\times K'}\abs{\sum_{i=1}^nf_i(y) h_i(z)}.
\]
\par
Hence 
$ f\tensor h \mapsto <Af,h>$
extends to a linear form on the injective tensor product
$C(k)\injtensor C(K')= C(K\times K')$, so there is a Borel measure
$\mu$ such that
for 
$f_1, \ldots ,f_n \in C(K)$ and $h_1, \ldots , h_n \in C(K')$
\[
\int_{G} \sum_{i=1}^nAf_i(z) h_i(z)\, dz=%
\int_{K}\int_{K'}\sum_{i=1}^n f_i(y) h_i(z) d\mu(y,z),
\]
and by a computation much similar to the first part of the above one
\[
\abs{\int_{K}\int_{K'}\sum_{i=1}^n f_i(y) h_i(z) d\mu(y,z)}
\leq 
\int_{G}\int_{G}a(zy^{-1})\Delta(y^{-1})\abs{\sum_{i=1}^n f_i(y) h_i(z)}\,%
dy dz.
\]
This last inequality extends to bounded measurable functions,
and if $H\subset K\times K'$ is the
characteristic function of a set of 
Haar measure $0$ in $G\times G$,
then by Fubinis theorem,
for almost all $z\in G$ the function $H(.\, ,z)$
vanishes off a set of measure zero.
The above estimate implies
\[0\leq \int H \, d\mu \leq %
\int_{G}\int_{G}a(zy^{-1})\Delta(y^{-1}){H}(y,z)\,dydz=0.\]
That is, $\mu$ is absolutely continuous with respect to
the Haar measure on $G\times G$, and by the Radon Nikodym theorem there
exists a kernel $F_A^{(K',K)}\in L^1(K\times K')$ such that
\[\int_{G} Af(x)h(x) \, dx = %
\int_{K'}\int_{K} F_A^{(K',K)}(x,y) f(y)h(x)\, dy dx,\]
whenever $f,h$ are continuous with support in $K$ resp. $K'$.
It is now a standard procedure to check the consistency of these
kernels for different pairs of compact sets so that they define a
locally integrable kernel $F_A$ on $G\times G$, which represents
$A$ as claimed in the proposition.
\end{proof}
\begin{rem}
In the above proof we have seen that in terms of the kernel $F_A$
of a convolution dominated operator $A$ the inequality
(\ref{def:cdop}) may be rewritten as
\begin{equation}\label{rem:cdpointwise}
F_A(x,y)\leq a(x,y^{-1})\Delta(y^{-1}),\quad  \mbox{ locally almost everywhere 
(l.a.e.) }
\end{equation}
\end{rem}
\begin{rem}
Conversely, if $F:G\times G \to \CC$ is a locally integrable
function, such that for some $a\in  L^1(G)$:
\[ \abs{F(x,y)}\leq a(xy^{-1})\Delta(y^{-1}),\]
then by 
\[Af(x)=\int_{G} F(x,y) f(y)\, dy\]
a bounded operator on $L^2(G)$ can be defined,
which clearly is dominated by $\lambda(a)$.
For this it suffices to check the Schur conditions
\[\essup_x \int \abs{F(x,y)}\,dy <\infty
\mbox{ and }
\essup_y \int \abs{F(x,y)}\,dx <\infty.\]
\end{rem}
\begin{rem}
The kernel of a convolution dominated operator satisfies the Schur conditions 
and hence represents the operator in the sense that for all
$f\in L^2(G)$ the following integral converges l.a.e. and
\[Af(x)=\int_{G} F_A(x,y) f(y)\, dy\quad \mbox{ l.a.e.}\]
\end{rem}
\par
Now let $a_i,\, i\in \NN$ be a sequence such that (\ref{def:cdnorm})
holds true for each $i$ and such that $\norm[1]{a_i}\leq \norm[cd]{A}+2^{-i}$.
Then, for $n\in \NN$, let $b_n=a_1\wedge \ldots \wedge a_n$,
where $\wedge$ denotes the operation of taking the pointwise minimum of
integrable functions. The functions $b_n$ are bounded below, and
form a decreasing sequence in $L^1(G)$. Thus they converge to $b$,
say. It is easily seen from (\ref{rem:cdpointwise}) that
this limit satisfies (\ref{def:cdop}), further $\norm[1]{b}=\norm[CD]{A}$.
We conclude:
\begin{rem}
The infimum in (\ref{def:cdnorm}) is attained.
\end{rem} 
\par
If $G$ is a discrete group then
an element $A\in CD$ may be  represented uniquely 
by its matrix with respect to
the basis given by the unit masses placed at the group elements
\[A(x,y)= (A \delta_y|\delta_x).\]
Denote  $m_z$ as the $z$-th
side-diagonal of the matrix. Then the matrix is the direct sum 
of its side diagonals and therefore
\[A= \sum_{z\in G} \lambda(z)\circ D^{m_z},\]
where $D^{m}\in B(l^2(G))$ is the multiplication operator
with $m\in l^{\infty}(G)$.
In \cite{fgl08}
this was used to show that 
\begin{equation}\label{eq:R}
R:\ell ^1(G,\ell ^{\infty}(G),T) \to B(\ell ^2(G))
\end{equation}
defined by
\begin{equation}
  \label{eq:cc6}
R:\sum _{z\in G} \delta_z \tensor m_z \mapsto \sum _{z\in G} \lambda(z)\circ D^{m_z}.
\end{equation}
is surjective from a certain twisted $L^1$-algebra onto $CD$.
In fact this map is an isometric $\ast$-algebra isomorphism.
We shall next define the twisted $L^1$-algebra and the map
$R$, which unfortunately is no longer surjective.
\par
Let $\Ruc(G)$ denote the space of bounded right uniformly continuous functions,
i.\ e.\ those $F\in L^{\infty}(G)$ such that 
$\essup_{x\in G} \abs{f(y^{-1}x) - f(x)} \to 0 \mbox{ as } y\to 0$. 
(This definition of right uniform continuity 
follows \cite[Ch. 3, 1.8(vi)]{rei67})
This space is just the subspace of $L^{\infty}$
of those elements on which left translation acts 
norm continuously. It is a closed subspace 
containing only continuous functions.
\par
For 
$y\in G$
denote $T_y$ left translation on $\Ruc(G)$, that is
$T_yn(z)=n(y^{-1}z)\; n\in Ruc(G)$.
We consider the map $T:y\mapsto T_y$ as a 
homomorphism of $G$ into the group of isometric automorphisms of the 
$C^{\ast}$-algebra $\Ruc(G)$,
which is continuous when the latter group 
is endowed with the strong operator topology.
With this homomorphism we form the twisted $L^{1}$-algebra
$\mathcal{L}=L^1(G,\Ruc(G),T)$ in
the sense of Leptin~\cite{lep67,lep67a,lep68}. The  underlying Banach
space of $\mathcal{L}$ is  the space of $\Ruc(G)$-valued 
Bochner integrable functions on $G$, but we will often interpret it  
as the projective tensor product 
\[L^1(G,\Ruc(G)) =  L^1(G) \, \projtensor \,  \Ruc(G).\]
Thus for an element $f\in L^1(G,\Ruc(G))$ we denote its value in
$\Ruc(G)$ by $f(x)$, $ x\in G$, and we  write $f(x)(z)$ or $f(x,z)$ for 
the value of this $\Ruc$-function at $z\in G$.
\par
The twisted convolution of $h,f \in \mathcal{L}$  is defined  by
\[
h \conv f (x) =  \int_{G} T_y h(xy) f(y^{-1})\,dy, \;\mbox{ for } x\in G \, ,
\]
and the  involution of $h\in \mathcal{L}$  by
\[
h^{\ast} (x) = \Delta(x^{-1})\overline{T_{x^{-1}} (h(x^{-1}))}, \;\mbox{ for } x\in
G\, .
\]
The properties of the projective tensor product ensure
that
$
R: a\tensor m \mapsto \lambda(a)\circ D^{m}
$ 
extends to a norm-nonincreasing linear map from
$L^1(G,\Ruc(G))$ to $B(L^2(G))$.
\begin{prop}\label{prop:R}
The map $R:L^1(G,\Ruc(G),T)\to B(L^2(G))$ is
an $\ast$-algebra homomorphism with range in $CD$.
It is isometric from $L^1(G,\Ruc(G),T)$ into $CD$.
\end{prop}
\begin{proof}
Since $\lambda$ is a continuous unitary representation of $G$ on $L^2(G)$
and $D:m \mapsto D^m$ is a $\ast$-representation of $\Ruc$ in $B(L^2(G))$, 
with $\lambda(x)^{\ast}\circ D^m \circ \lambda(x) = D^{T_x^{-1}m}$,
$\forall x \in G ,\, m\in \Ruc$, \cite[Satz 3]{lep68} shows
 that $R$ defines a non-degenerate
$\ast$-representation of $\mathcal{L}$.
\par
For $a\tensor m \in L^1(G) \, \projtensor \,  \Ruc(G)$ and $h\in L^2(G)$
we have
\[
\abs{Rf(x)} = \abs{\lambda(a)(mh)(x)} \leq \norm[\infty]{m} \abs{a}\conv \abs{f}(x).
\]
This shows that $R$ maps into CD and does not increase the norm as a map
from $\mathcal{L}$ to $CD$.
\par
Now assume that for $f\in L^1(G,\Ruc(G),T)$ we have $a\in L^1(G)$
such that for all $h\in L^2(G)$: $\abs{Rf(h)}(x)\leq a\conv\abs{h}(x)$. Then
\begin{eqnarray*}
\int_G a(xy^{-1}) \Delta(y^{-1}) \abs{h(y)}\,dy &=&
a\conv\abs{h}(x) \\
&\geq &\abs{Rf(h)}(x)\\
&= &\abs{\int_G \lambda(y)(f(y)(.)h(.))(x)dy}\\
&=& \abs{\int_G f(y)(y^{-1}x)h(y^{-1}x)dy}\\
&=&\abs{\int_G f(xy^{-1})(y)\Delta(y^{-1})h(y)\; dy}.
\end{eqnarray*}
Then
$\abs{f(xy^{-1})(y)}\leq a(xy^{-1})\Delta(y^{-1})$ for almost all 
$(x,y) \in G \times G$, or
$\abs{f(x)(y)}\leq a(x)$ for almost all 
$(x,y) \in G \times G$.
It follows that
$\norm[\infty]{f(x)} \leq a(x)$ for almost all $x\in G$.
Hence 
\[\norm[L^1(G,\Ruc(G),T)]{f}=\int_G \norm[\infty]{f(x)}\,dx\leq \norm[1]{a}.\]
As a consequence: $R:L^1(G,\Ruc(G),T) \to CD$ is isometric.
\end{proof}
\begin{defi}\label{def:cdreg}
Elements in the image of $R$ we call regular convolution dominated operators,
and denote the whole image by $CD_{reg}$.
\end{defi}
The continuous functions vanishing at infinity $C_0(G)$ are a closed two-sided
 ideal in
$\Ruc(G)$ and as is easily seen this implies that 
$L^1(G,C_0(G),T)$ is a closed two-sided ideal in $L^1(G,\Ruc(G),T)$.
\begin{rem}
It follows from \cite[Theorem 4]{LP79} that $L^1(G,C_0(G),T)$
is simple and symmetric. The representation $\rho$ in the beginning of their
proof, we denoted it by $R$, maps $L^1(G,C_0(G),T)$ into an ideal (in $CD_{reg}$)
of
compact operators. Moreover the operator norm closure
of $R(L^1(G,C_0(G),T))$ equals the compact operators~\cite{lep79}.
\end{rem}

\section{Symmetry of the twisted $L^1$-algebra}

In this section we shall show that the twisted $L^1$-algebra
$L^1(G,\Ruc(G),T)$ is a symmetric Banach $\ast$-algebra.
To this end we shall first recall a criterion for the symmetry of a 
Banach $\ast$-algebra.
\par
\begin{defi}\label{def:preunitary}
Let $E$ be a normed linear space and $\mathcal{A}$ be a Banach $\ast$-algebra.
A representation $\rho:\mathcal{A}\to \END(E)$ is called
preunitary if there exists a Hilbert
space $H$ and a bounded $\ast$-representation $\pi:\mathcal{A}\to B(H)$
together with an injective and bounded  operator
$U:E\to H$ intertwining $\rho$ and $\pi$.
\[
U\circ \rho(a)= \pi(a)\circ U,\quad \forall a \in \mathcal{A}.
\]
\end{defi}
\begin{rem}
If $\rho$ is a contractive representation of $A$ on a Banach space $E$
then we are given a Banach $A$ module in the sense of Leptin.
The representation is preunitary in the above sense if the Banach
$A$ module is preunitary in the sense of Leptin~\cite{Lep77}.
\end{rem}
It is clear that the image, under $U$, of a 
$\rho(\mathcal{A})$-invariant subspace is invariant under
$\pi(\mathcal{A})$. We may and do assume that $U(E)$ is dense in H.\\
\par
The following question appears naturally.
If $\rho$ is an algebraically irreducible, preunitary
representation, can the representation $\pi$ in the definition
be chosen topologically irreducible?
The answer is positive:
\begin{prop}
Let $A$ be a Banach $\ast$-algebra
with approximate identity. 
If $\rho : A \to B(E)$ is an  algebraically irreducible,
preunitary representation of $A$ 
then there is a topologically irreducible 
 $\ast$ representation $\pi$
extending  $\rho$.
\end{prop}
\begin{proof}
The representation $\rho$ is  preunitary, hence
its kernel is a $\ast$-ideal, and possibly
replacing $A$ by $A/\mbox{kern}(\rho)$ we may assume that 
$\rho$ is faithful.
Since $\rho : A \to B(E)$ is algebraically irreducible
there is a maximal modular left ideal $M \subset A$,
with modular right unit $u$, such that 
$E$ and $A/M$ are algebraically isomorphic, and $\rho$
appears as left multiplication on $A/M$. For $b\in A$ denote $\overline{b}$
its class in $A/M$ and let
\[
(\;~\; |\;~\;):A/M \times A/M \to\CC
\]
be the positive sesquilinear form given by
\[
(\overline{b}|\overline{c}) = (U(\overline{b})|U(\overline{c}))_{H}.
\]
The functional
\[
\phi(a) = (a\overline{u}|\overline{u}) = 
(\pi(a)U(\overline{u})|U(\overline{u}))_H ,\quad a\in A
\]
is non-trivial, positive and continuous with respect to
the maximal $C^{\ast}$-norm of $A$.
\par
Let $C$ denote the $C^{\ast}$-hull of $A$
and $\overline{M}$ the closure of $M$ in $C$.
For $a\in M$ we have $\phi(a)=0$ hence $\overline{M}\subset \mbox{kern}(\phi)$,
so $\overline{M}\noteq C$. Let $N \supset M$ be a proper maximal modular
left ideal in $C$ containing $M$. 

First we claim that $N\cap A= M$.
By definition of $N$ we only have to show
$N\cap A\subset M$. Clearly $N\cap A$ is a left ideal in $A$, and
$u$ a right modular unit. Note that $N\cap A$ is a proper
ideal in $A$ since $N$ is a proper ideal in $C$.
By the maximality of $M$ it follows that $N\cap A\subset M$.
\par
Now by \cite[2.9.5]{Dix79} there is a pure state $\psi$ on $C$ such that
$N=\{\,b\in C\,:\, \psi(b^{\ast}b)=0\,\}$.
Hence $M=N\cap A= \{\,a \in A\,:\,\psi(a^{\ast}a)=0\,\}$.
So the GNS representation of $A$ constructed from the (pure state)
$\psi_{|A}$ is a topologically irreducible extension of $\rho$
containing $E$ as a dense invariant subspace. 
\end{proof}
\par
Since an algebraically irreducible representation of a Banach $\ast$-algebra is
equivalent to a contractive one, we may use from \cite{lep76a}:
\begin{thm}\label{theorem:preunitary}
A Banach $\ast$-algebra is symmetric if and only if
all its non-trivial algebraically irreducible representations are preunitary.
\end{thm}
\par
Following the concept used in the discrete case~\cite{fgl08}
one would like to define a map
\[
Q:L^1(G,\Ruc(G),T)\to L^1(G)\projtensor B(L^2(G))
\]
by
\[
Q(f) = \{x\mapsto \lambda(x)\circ D^{f(x)}\}.
\]
But this does not work since  $Q(f)$ is not Bochner measurable because
$\lambda:G\to B(L^2(G))$, is strongly continuous, but not norm continuous.
\par
The problem can be worked around by showing, completely analogously to 
\cite{pog92}, that algebraically irreducible representations
of certain twisted $L^1$-algebras of $G$ remain irreducible
when ''restricted`` to the discretised group.
\par
So, as before let $G$ be a locally compact group, $\mathcal{A}$ a Banach 
$\ast$-algebra, with isometric involution and a left approximate identity.
Further we assume that $T:G\to \Aut(\mathcal{A})$ is a continuous homomorphism 
from $G$ into the 
group of $\ast$-automorphisms of 
$\mathcal{A}$, where $\Aut(\mathcal{A})$  is endowed with the strong
operator topology, i.e.\ $y\mapsto T_y a$ is continuous
from $G$ to $\mathcal{A}$, for all $a\in \mathcal{A}$.
With these data we form the twisted $L^1$-algebra
(as above) $\mathcal{L}=L^1(G,\mathcal{A},T)$.
\par
Let $E$ be a nontrivial linear space and let $\rho:\mathcal{L}\to B(E)$ 
be a non-trivial algebraically irreducible representation on it.
Given $\xi_0\in E$, $\xi_0\neq 0$, one has a norm on $E$:
\begin{equation}\label{eq:norm}
\norm[E]{\xi}= \inf \{\norm[\mathcal{L}]{f}\, :%
\, f\in \mathcal{L},\; \rho(f)\xi_0=\xi \},
\end{equation}
with respect to which $E$ is a complete space. In fact it is the 
quotient of $\mathcal{L}$ with respect to the maximal modular left ideal
$\{ f\in \mathcal{L}\,:\; \rho(f)\xi_0=0 \}$.
Different $\xi\in E$ define different but equivalent norms.
\par
As in the proof of \cite[Satz 3]{lep68},
$\rho:G \to B(E)$ and
$\rho:\mathcal{A}\to B(E)$  are representations  of the group
respectively of the 
Banach algebra $\mathcal{A}$.
The operators 
$\rho(x)$, $x\in G$, are isometries and
$\rho$ does not increase norms.
Here the operators $\rho(x), \, x\in G$ and $\rho(a),\,a\in \mathcal{A}$ 
do not necessarily commute, 
but we have the relation
\[
\rho(y^{-1})\rho(a)\rho(y)=\rho(T_{y^-1}a),
\]
and furthermore for all $f\in L^1(G,\mathcal{A},T)$, $\xi\in E$
\[
\rho(f)\xi = \int_G \rho(x) \rho(f(x)) \xi \, dx.
\]
Unfortunately the proof in \cite{lep68} is done with the hypothesis of 
dealing with a $\ast$-representation. Apart from some 
algebraic identities the main ingredient is \cite[Satz 2]{lep68}
in its consequence (1.2) loc.cit.\ .
\par 
Now we take the group $G$ with the discrete topology, denote it $G_d$.
We form the twisted $L^1$-algebra $\mathcal{L}_d:=l^1(G,\mathcal{A},T)$,
and define a representation of it on $E$ by
\[
\rho_d(h)\xi = \sum_{x\in G} \rho(x) \rho(h)\xi,%
\quad \xi\in E, \; h\in \mathcal{L}_d.
\]
\begin{lem}\label{lem:discrete}
Let $\rho$ be an algebraically irreducible representation of 
$L^1(G,\mathcal{A},T)$
and assume the above settings.
Then the representation $\rho_d:l^1(G,\mathcal{A},T)\to B(E)$
is algebraically irreducible.
\end{lem}
\begin{proof}
We follow the proof of \cite[Theorem 2]{pog92}.
Assume that $E'\subset E$ is a non-trivial $\rho_d(\mathcal{L}_d)$
invariant subspace; we have to show that $E' = E$.
To this end we take a fixed nonzero $\xi_0\in E'$,
and the corresponding norm, see (\ref{eq:norm}), on $E$.
\begin{claim}
We claim that for $\xi\in E$ and $\varepsilon>0$
there is
$h\in l^1(G,\mathcal{A},T)$ such that
\begin{equation}\label{eq:approx}
\norm[E]{\rho_d(h)\xi_0 -\xi}\leq\varepsilon%
\mbox{ and }
\norm[\mathcal{L}_d]{h}\leq \norm[E]{\xi} +\varepsilon.
\end{equation}
\end{claim}
The claim implies the assertion of the lemma.
For, if $\eta\in E$ is given we have to find $h\in \mathcal{L}_d$
such that $\rho_d(h)\xi_0 = \eta$, and this is done inductively as follows:
First we find $h_1\in\mathcal{L}_d$ such that
\begin{equation*}
\norm[E]{\rho_d(h_1)\xi_0 -\eta}\leq 2^{-1}%
\mbox{ and }
\norm[\mathcal{L}_d]{h_1}\leq \norm[E]{\eta} +2^{-1}.
\end{equation*}
If $h_1,\ldots,h_n$ are already defined with
\begin{equation*}
\norm[E]{\sum_{i=1}^{n}\rho_d(h_i)\xi_0 -\eta}\leq 2^{-n}
\end{equation*}
and
\begin{equation*}
\norm[\mathcal{L}_d]{h_i}\leq \norm[E]{\sum_{j=1}^{i-1}\rho_d(h_j)\xi_0 -\eta} +2^{-i}, \quad i=1,\ldots,n,
\end{equation*}
then we choose $h_{n+1}\in \mathcal{L}_d$ such that
\begin{equation*}
\norm[E]{\rho_d(h_{n+1})\xi_0-\left(\eta-\sum_{i=1}^{n}\rho_d(h_i)\xi_0\right)}\leq 2^{-(n+1)}
\end{equation*}
and
\begin{equation*}
\norm[\mathcal{L}_d]{h_{n+1}}\leq \norm[E]{\left(\sum_{i=1}^{n}\rho_d(h_i)\xi_0 -\eta\right)} +2^{-(n+1)}.
\end{equation*}
Since 
\[
\norm[\mathcal{L}_d]{h_i}\leq \norm[E]{\sum_{j=1}^{i-1}\rho_d(h_j)\xi_0 -\eta} +2^{-i}\leq 2^{i-1}+2^{-i}
\mbox{ for }i\geq 2,
\] the sum $\sum_{1}^{\infty} h_i$, call it $h$,
exists in $\mathcal{L}_d$, and
\[
\norm[E]{\rho_d(h)\xi_0 -\eta} = \lim_{n\to \infty}\norm[E]{\sum_{i=1}^{n}\rho_d(h_i)\xi_0 -\eta}=0.
\]
It remains to establish the claim.
So let $\delta>0$ be a positive real number to be determined later.
By definition of the norm on $E$ we find
$f\in L^1(G,\mathcal{A},T)$, with
\[
\rho(f)\xi_0=\xi \mbox{ and } \norm[\mathcal{L}]{f}< \norm[E]{\xi} +\delta.
\]
Since the space of continuous, compactly supported, $\mathcal{A}$
valued functions
$C_{cp}(G,\mathcal{A})$
is dense in $L^1(G,\mathcal{A},T)$
we find, in turn, a function $f_1\in C_{cp}(G,\mathcal{A})$
such that
\[\norm[E]{\rho(f_1)\xi_0 -\xi}<\delta%
\mbox{ and }
\norm[\mathcal{L}]{f_1}< \norm[E]{\xi} +\delta.
\]
Denote $S$ the support of $f_1$ and $\abs{S}$ its Haar measure.
Since $\rho$ is a strongly continuous representation of $G$
and since $f$ is uniformly continuous, there is a 
neighbourhood $U$ of the identity such that
\[
\norm[E]{\rho(u)\xi_0 -\xi_0}<\delta,\; \forall u\in U \mbox{ and } \]
\[\norm[\mathcal{A}]{f_1(xu) -f_1(x)}<\delta\abs{S}^{-1},\; \forall x\in G,\,u\in  U.
\]
As $S$ is compact it can be covered by finitely many translates
$x_1U,\ldots , x_mU$ of $U$. We make this covering disjoint
$V_1:=x_1U\cap S$ and inductively $V_k:=(x_kU\cap S)\setminus \cup_{j<k} V_j$.
The $V_j$ are measurable pairwise disjoint subsets of $S$,
hence $\sum_1^m \abs{V_j}\leq \abs{S}$.
\par
Now let $a_j=f_1(x_j)\in \mathcal{A}$, $j=1,\ldots,m$, 
and $f_2\in L^1(G,\mathcal{A},T)$ be given by
\[
f_2= \sum_{j=1}^{m} a_j \chi_{V_j}.
\]
Then
\begin{eqnarray*}
\norm[L^1(G,\mathcal{A},T)]{f_2}&\leq& 
\norm[L^1(G,\mathcal{A},T)]{f_2-f_1}+\norm[L^1(G,\mathcal{A},T)]{f_1}\\
&\leq&\sum_{j=1}^{m}\int_{V_j}\norm[\mathcal{A}]{f_1(x)-f_1(x_j)}\, dx
+ \norm[E]{\xi} +\delta\\
&\leq&\sum_{j=1}^{m} \abs{V_j}\delta\abs{S}^{-1}+ \norm[E]{\xi} +\delta\\
&\leq&\norm[E]{\xi} +2\delta.
\end{eqnarray*}
The desired  $h\in l^1(G,\mathcal{A},T)$ is defined by
\[
h=\sum_{j=1}^{m} a_j \abs{V_j} \delta_{x_j}.
\]
Then
\begin{eqnarray*}
\norm[l^1(G,\mathcal{A},T)]{h}&=&\sum_{j=1}^{m} \norm[\mathcal{A}]{a_j} \abs{V_j}
\;=\;\norm[L^1(G,\mathcal{A},T)]{\sum_{j=1}^{m} a_j \chi_{V_j}}\\
&\leq&\norm[E]{\xi} +2\delta.
\end{eqnarray*}
Moreover,
\[
\norm[E]{\rho_d(h)\xi_0-\xi}\leq\norm[E]{\rho_d(h)\xi_0-\rho(f_2)\xi_0}%
+\norm[E]{\rho(f_2)\xi_0-\xi}
\]
The second term can be estimated by
\begin{eqnarray*}
\norm[E]{\rho(f_2)\xi_0-\rho(f_1)\xi_0}+\norm[E]{\rho(f_1)\xi_0-\xi}
&\leq& \norm[L^1(G,\mathcal{A},T)]{f_1 -f_2}\norm[E]{\xi_0}+ \delta\\ 
&\leq& \delta\norm[E]{\xi_0}+\delta.
\end{eqnarray*}
For the first  term we use that $\rho:\mathcal{A}\to B(E)$ is bounded by one
and that $s\in V_j$ can be written $s=x_j u$ with $u\in U$:
\begin{eqnarray*}
\norm[E]{\rho_d(h)\xi_0-\rho(f_2)\xi_0}&\leq&%
\sum_{j=1}^{m}\norm[E]{%
\rho(a_j)\abs{V_j}\rho(x_j)\xi_0 - \int_{V_j} \rho(a_j)\rho(s)\xi_0 \,ds}\\
&\leq&\sum_{j=1}^{m}\norm[\mathcal{A}]{a_j}%
\norm[E]{\int_{V_j}\rho(x_j)\xi_0 - \rho(s)\xi_0\,ds}\\
&\leq&\sum_{j=1}^{m}\norm[\mathcal{A}]{a_j}%
\int_{V_j}\norm[E]{\rho(x_j)\xi_0 - \rho(s)\xi_0}\,ds\\
&\leq&\sum_{j=1}^{m}\norm[\mathcal{A}]{a_j}\int_{V_j}\delta\,ds\\
&=&\delta \norm[l^1]{f_2}\;\leq\;\delta(\norm[E]{\xi}+2\delta)
\end{eqnarray*}
Altogether we found $h\in l^1(G,\mathcal{A},T)$ such that
\[\norm[l^1(G,\mathcal{A},T)]{h}\leq\norm[E]{\xi} +2\delta
\mbox{ and }
\norm[E]{\rho_d(h)\xi_0-\xi}\leq\delta(\norm[E]{\xi}+2\delta);
\]
taking $\delta$ small enough now proves the claim.
\end{proof}
From here onward we assume that $\mathcal{A}$ is a $C^\ast$-algebra 
and that the operators $T_y$ are isometries. We recall that we assumed
that they  preserve the involution:
$T_y (a^{\ast}) = (T_y a)^{\ast}$, $a\in \mathcal{A},\,y\in G$.
Before discussing the symmetry of $L^1(G,\mathcal{A},T)$
we shall first look at the discretised version.
\par
So let $D:\mathcal{A}\to B(\mathcal{H})$ be a faithful $\ast$-representation
of $\mathcal{A}$ on some Hilbert space $\mathcal{H}$.
We define a map 
\begin{equation}\label{eq:Q}
Q:\ell ^1(G_d,\mathcal{A},T) \to \ell ^1(G_d) \projtensor B(\mathcal{H})
\end{equation}
by 
\begin{equation}
  \label{eq:c23}
  f=\sum_v \delta_v \otimes m_v \mapsto \sum_v \delta_v \otimes
  T_{v}\circ D(m_v)\, .
\end{equation}
\begin{prop}\label{p:2}
The  map $Q$ is an isometric $\ast$-iso\-mor\-phism of 
$\ell ^1(G_d,\mathcal{A},T)$ onto a closed $\ast$-subalgebra of $l^1(G_d) \projtensor
B(\mathcal{H})$.
\end{prop}
\begin{proof}
The proof rests on the isometric identification 
$l^1(G,E)=l^1(G)\, \projtensor \,  E$, which holds for any Banach space $E$~\cite[Ch.~VIII.1.10]{diuh77}. 
It follows that
for $f=\sum_v \delta_v\otimes m_v \in l^1(G,\mathcal{A},T)$
\begin{eqnarray*}
\norm[1]{f} &=& \sum_v \norm[\mathcal{A}]{m_v}\; = \; 
\sum_v\norm[B(\mathcal{H})]{T_{v}\circ D(m_v)}\\
&=& \norm[\ell ^1(G)\projtensor B(\mathcal{H})] {\sum_v \delta_v\otimes T_{v}\circ D(m_v)}.
\end{eqnarray*}
Thus $Q$ is an isometry. 
Let $h=\sum_v \delta_v\otimes n_v$,  then 
\[
h\conv f = \sum_v \delta_v\otimes l_v,
\]
where
$l_v= \sum _{y\in G} (T_y n_{vy} )m_{y^{-1}}$. Hence
\begin{eqnarray*}
Q(h\conv f) &=& \sum_v \delta_v\otimes T_vD({l_v})\\
&=&\sum_v \delta_v\otimes\sum_{\{z,w:zw=v\} }T_zD({n_z}) T_wD({m_w})\\
&=& \sum_{z,w }\delta_z \delta_w \otimes T_zD({n_z})T_w D({m_w})\\
&=&(\sum _{z\in G} \delta_z \otimes T_zD({n_z}) )%
(\sum_w\delta_w \otimes T_wD({m_w}))%
\;=\; Q(h)  Q(f)\, .
\end{eqnarray*}
Similarly one computes that $Q$ intertwines the involutions.
In fact
\begin{eqnarray*}
  Q(f)^{\ast}&=&\sum_v \delta_v^{\ast}\otimes (T_{v}\circ D(m_v))^{\ast}\\
&=&\sum_v \delta_{v^{-1}}\otimes T_{v^{-1}}D({T_v {m_v}^\ast})\\
&=&\sum_{v^{-1}} \delta_{v}\otimes T_{v}D({T_{v^{-1}}
  {m_{v^{-1}}}^\ast})  = Q(f^*) \, .
\end{eqnarray*}
Thus $Q$ is a $\ast $-homomorphism. Since $Q$ is an  isometry, the
image of $Q$ is a closed subalgebra of 
$l^1 (G_d)\, \projtensor \, B(\mathcal{H})$. 
\end{proof}
Because symmetry is inherited by   closed subalgebras,  we   obtain the
following consequence. 
\begin{cor}\label{cor:1}
Let $G$ be a discrete rigidly symmetric group, $\mathcal{A}$ a $C^*$-algebra.
 Then 
 $l^1(G_d,\mathcal{A},T)$ is a symmetric Banach $\ast$-algebra.
Especially any of its algebraically irreducible representations is preunitary.
\end{cor}

\begin{thm}
If $G$ is rigidly symmetric as a discrete group, and if 
$\mathcal{A}$ is a $C^\ast$-algebra, then
$L^1(G,\mathcal{A},T)$ is symmetric.
\end{thm} 
\begin{proof}
We shall verify that a non-trivial algebraically 
irreducible representation
$\rho:L^1(G,\mathcal{A},T)\to \END(E)$ is preunitary. We know
that its discretised version
$\rho_d:l^1(G_d,\mathcal{A},T)\to \END(E)$ is preunitary too.
So let $H$ be the Hilbert space, $\pi_d$ the $\ast$- representation of
$l^1(G_d,\mathcal{A},T)$ on it, and $U$ the intertwining operator
according to the definition~\ref{def:preunitary}. 
Since $E$ is a complete space with respect
to the norm 
\[
\norm[d]{\xi}= \inf \{\norm[l^1(G_d,\mathcal{A},T)]{h}\, :%
\, h\in l^1(G_d,\mathcal{A},T),\; \rho(h)\xi_0=\xi \},
\]
as well as with respect to $\norm[E]{}$, these norms are equivalent.
Furthermore, 
$h\mapsto \pi_d(h) U \xi_0 = U \rho_d(h) \xi_0$
is bounded from 
$l^1(G_d,\mathcal{A},T)$ to $H$. These two facts show that $U : E \to H$
is  bounded.
\par
From \cite[Satz 3]{lep68}
we know that there is a unitary representation
$\pi : G_d \to B(H)$ and a $\ast$-representation
(again denoted by the same letter) $\pi:\mathcal{A}\to B(H)$
such that
\[
\pi(x^{-1})\pi(a)\pi(x)=\pi(T_{x^{-1}} a),\quad%
 \forall x\in G,\, a\in \mathcal{A}
\]
and
\[
\pi_d(h) = \sum_{x\in G} \pi(x)\pi(h(x)),\quad%
\forall h\in l^1(G_d,\mathcal{A},T).
\]
\par
Since $\rho:G\to B(E)$ is a continuous representation
and $U$ a continuous intertwining operator
it follows that by means of $\pi$ the group
$G$ acts continuously on the image $U(E)$ in $H$.
As this subspace is dense in $H$ and $\pi$ is a bounded representation
we infer that this action is continuous on the whole space $H$.
This allows to define a representation
$\pi$ of $\mathcal{L}=L^1(G,\mathcal{A},T)$ on $H$ by
\[
\pi(f)\eta = \int_G \pi(x)\pi(f(x))\eta \,dx\,\quad \eta\in H, f\in \mathcal{L}.
\]
\par
It is easily checked that this formula defines
a $\ast$-representation of $\mathcal{L}$ on $H$,
and that $U$ intertwines $\pi$ and $\rho$.
\end{proof}
\begin{cor}
If $G$ is rigidly symmetric as a discrete group
then the twisted algebra 
$L^1(G,\Ruc(G),T)$ is a symmetric Banach $\ast$-algebra.
\end{cor}
\begin{rem}
\begin{itemize}
\item[(i)]
Locally compact nilpotent groups are rigidly symmetric,
even as discrete groups~\cite[Corollary 6]{pog92}.
\item[(ii)]
If for the settings of
~\cite[Theorem 5]{LP79}
we choose $G=H=\RR$ with the action $\omega :(s,x)\mapsto x^s:=e^sx$
of the additive group $\RR$ on itself and $\mathcal{D}=C_0(\RR)$,
then we see that $L^1(\RR,\mathcal{D})$, with the trivial action
of $\RR$ on $\mathcal{D}$
 is symmetric, whereas  
$L^1(\RR,L^1(\RR,\mathcal{D}),\tilde{T})$, with action
$$\tilde{T}(s)f(x)(\cdot) = e^{-s}f(\omega(-s,x))(\omega(s,\cdot))
\in \mathcal{D}, \quad
f\in L^1(\RR,\mathcal{D}), s,x\in \RR$$
is not symmetric. Hence in the above theorem it is necessary to 
assume that $\mathcal{A}$ is a $C^*$-algebra and not only a symmetric Banach-$*$-algebra. See the example in section ~\ref{sec:example}.
\item[(iii)]
A special role in the theory of symmetry of 
group algebras is played by the group of affine mappings of the real line,
 the ``$ax+b$''-group, denoted $\AF$. This group has a symmetric $L^1$-algebra~\cite{lep76},
 but
its discretised version is not symmetric~\cite{jen69}. We do not know about
the rigid symmetry of the continuous group. 
\end{itemize}
\end{rem}
\section{Spectral Invariance of $CD_{reg}$}
We shall show that the spectrum of an element of
$CD_{reg}$ is the same no matter if it is
considered as a bounded operator 
on $L^2(G)$ or as an element of $CD_{reg}$.
To this end we define two representations
of $L^1(G,\Ruc(G),T)$ and first show that they are weakly equivalent.
The first representation is $R:L^1(G,\Ruc(G),T)\to CD_{reg}\subset B(L^2(G)$,
which we call the canonical representation.
The second one is the $D$-regular representation
$\lambda^D:L^1(G,\Ruc(G),T)\to B(L^2(G,L^2(G))$,
acting on 
$L^2(G,L^2(G))$.
\begin{prop}
The representations $\lambda^D$ and $R$ of
$L^1(G,\Ruc(G),T)$ are weakly equivalent, i.e.\
$\norm[B(L^2(G,L^2(G))]{\lambda^D(f)}=\norm[B(L^2(G)]{R(f)}$.
\end{prop}
\begin{proof}
We identify $L ^2(G,L ^2(G))$ with $L ^2(G\times G)$. Let
$R^\omega $ be the extension of $R$ from $L ^2(G)$ to $L^2(G\times G)$
  by letting the operators 
$R(f)$, 
$f\in \mathcal{L}$, 
act in the first coordinate only, i.e., for $\xi \in L ^2(G\times
G)$ 
\begin{equation}
  R^{\omega}(f) \xi (x,z) =\int_{G} f(y)(y^{-1}x)\xi(y^{-1}x,z)\, dy.
\end{equation}
Next we  define
a  candidate for an intertwining operator between the  $D$-regular
representation and the  $\card{G}$-multiple $R^\omega $ of the canonical representation
by
\[
S\xi(x,z)=\xi(xz,z),\;\mbox{ where }\xi\in L ^2(G\times G).
\]
Then on the one hand we have 
\[
S[R^{\omega}(f) \xi](x,z) = \int_{G} f(y)(y^{-1}xz)\xi(y^{-1}xz,z)\, dy.
\]
 On the other hand
\begin{eqnarray*}
\lambda^D(f) (S\xi) (x,z) &=& \int_G (T_yf(xy))(z) (S\xi)(y^{-1},z)\, dy\\
&=&\int_G (T_{x^{-1}y}f(y))(z) (S\xi)(y^{-1}x,z)\, dy\\
&=&\int_G  f(y)(y^{-1}xz) (S\xi)(y^{-1}x,z)\, dy\\
&=& \int_G f(y)(y^{-1}xz) \xi(y^{-1}xz,z)\, dy.
\end{eqnarray*}
Consequently, 
\begin{equation}
  \label{eq:cc9}
\lambda ^D(f) (S\xi ) = SR^\omega (f)\xi   
\end{equation}
 for all $f\in
\cL $ and $\xi \in L ^2(G\times G)$. Since $S $ is unitary on $L ^2(G\times G)$, $\lambda ^D$ and
$R^\omega $ are equivalent, whence 
$\norm{R^\omega(f)}=\norm{\lambda ^D(f)}$ for all $f\in \mathcal{L}$. 
\par
Now for $\xi \in L^2(G\times G)$ write $\xi_z(x)=\xi(x,z)$:
\begin{eqnarray*}
\norm{R^{\omega}(f) \xi}^2 &=&\int_G \int_G %
\abs{\int_{G} f(y)(y^{-1}x)\xi(y^{-1}x,z)\, dy}^2\,dx dz\\
&=&\int_G \int_G %
\abs{\int_{G} f(y)(y^{-1}x)\xi_z(y^{-1}x)\, dy}^2\,dx dz\\
&=&\int_G \int_G %
\abs{R(f)\xi_z(x)}^2\,dx dz\\
&\leq&\int_G\{\norm{R(f)}\norm{\xi_z}\}^2\,dz\\
&=&\norm{R(f)}^2\norm{\xi}^2.
\end{eqnarray*}
it follows that $\norm{R^{\omega}(f)}\leq\norm{R(f)}$.
\par
For the converse inequality let $\varphi\in L^2(G)$
be an element of norm one. Embed $L^2(G)$ in $L^2(G\times G$
by $\xi \mapsto\xi'$ where $\xi'(x,z)=\xi(x)\varphi(z)$ then
\begin{eqnarray*}
\norm{R(f)\xi}^2&=&\int_G %
\abs{R(f)\xi(x)}^2\,dx\\
&=&\int_G \int_G \abs{R(f)\xi(x)}^2\abs{\varphi(z)}^2\,dxdz\\
&=&\int_G \int_G\abs{R^{\omega}(f)\xi'(x,z)}^2\,dxdz\\
&\leq&\norm{R^{\omega}(f)}^2\int_G \int_G\abs{\xi'(x,z)}^2\,dxdz\\
&=&\norm{R^{\omega}(f)}^2\norm{\xi}^2.
\end{eqnarray*}
\end{proof}
\begin{cor}
Let $G$ be an amenable group, which is rigidly symmetric as a discrete group.
Then for $f\in \mathcal{L}$
\[
\sprad[\mathcal{L}]{f^{\ast}f}=\norm[B(L^2(G))]{R(f)}^2.
\]
\end{cor}
\begin{proof}
We imposed the amenability on $G$ to have that the 
largest $C^{\ast}$-norm, denoted $\norm[\ast]{\,.\,}$,
 on $\mathcal{L}$ is just 
given by the $D$-regular representation, since
the representation $D$ of $\Ruc$ is a maximal 
representation~\cite[Satz 6]{lep68}. Therefore
\[
\norm[\ast]{f}=\norm{\lambda^D(f)}=\norm[B(L^2(G))]{R(f)}
\quad \forall f\in\mathcal{L}.
\]
Since $\mathcal{L}$ is symmetric Ptaks theorem~\cite{ptak70}
asserts
\[
\sprad[\mathcal{L}]{f^{\ast}f}=\norm[\ast]{f^{\ast}f}=\norm[\ast]{f}^2.
\]
\end{proof}
\begin{thm}
Let $G$ be an amenable group, which is rigidly symmetric as a discrete group.
Then for an operator $A\in CD_{reg}$
\[
\spr[CD]{A}=\spr[B(L^2(G))]{A}.
\]
\end{thm}
\begin{proof}
Since $A=R(f)$ for some $f\in \mathcal{L}$, 
this follows from the above corollary by an application of
Hulanicki's Lemma~\cite{fgl}.
\end{proof}

\section{An example of Non  Symmetry}\label{sec:example}
A special role in the theory of symmetry of 
group algebras is played by the group of affine mappings of the real line,
 the ``$ax+b$''-group. This group has a symmetric $L^1$-algebra~\cite{Lep77},
 but
its discretised version is not symmetric~\cite{jen69}. We do not know about
the rigid symmetry of the continuous group. 
We shall consider
the connected component of the identity.

\[ \AF=\left\{\begin{pmatrix}a&b\\ 0&1 \end{pmatrix}, \quad a>0, b\in \RR\right\}.\]

The multiplication is 
$(a,b)\cdot (a',b')=(a a', ab'+b)$ and the action on $\RR$:
$(a,b):x\mapsto ax+b$. The left Haar measure is $\frac{da}{a^2}db$.
\par
\begin{thm}\label{theorem:notsym}
For this group
$L^1(\AF,\Ruc({\AF}),T)$ is not symmetric, where $T$ is left translation on 
the right uniformly continuous functions $\Ruc({\AF})$.
\end{thm}
\begin{proof}
We identify the normal subgroup $N$ of translations with $\RR$,
and by means of the exponential map identify $\RR$ 
with the subgroup of dilations.
\par
Explicitly:
Let $\omega:(s,x)\mapsto e^sx$ from $\RR\times\RR\to \RR$.
The action of $\RR$ on $C_0(\RR)$ is given by 
$u\mapsto u^s:=u(\omega(s,\cdot))$.
The additive group $\RR$ acts on 
$L^1(\RR,C_0(\RR))$ by
\[\tilde{T}(s)f(x)(\cdot) = e^{-s}f(\omega(-s,x))(\omega(s,\cdot))\in C_0(\RR), 
\quad f\in L^1(\RR,C_0(\RR)), s,x\in \RR.\]
Here $e^s$ is the modulus of the action $\omega(s,\cdot)$
on $\RR$, with respect to translation invariant Lebesgue measure $dx$.
From~\cite[Theorem 5]{LP79} we know that
the twisted algebra $L^1(\RR,L^1(\RR,C_0(\RR)),\tilde{T})$ is not symmetric.
So the theorem will be proved if we show that this Banach-{$\ast$} algebra is
isomorphic to a closed $\ast$-subalgebra of $L^1(\AF,\Ruc({\AF}),T)$.
\par
We let $\AF$ act on $C_0(\RR)$ by $T'((a,b)):u\mapsto u\circ\omega(\log(a),\cdot)$. Then
$L^1(\AF,C_0(\RR),T')$ and $L^1(\RR,L^1(\RR,C_0(\RR)),\tilde{T})$
are isometrically isomorphic Banach *-algebras:
For $f\in L^1(\RR,L^1(\RR,C_0(\RR)),\tilde{T})$
let $Sf\in L^1(\AF,C_0(\RR),T')$ be defined
by $Sf(a,b)=f(\log(a))(-\frac{b}{a})\in C_0(\RR),\quad (a,b)\in\AF$.
Indeed
\begin{equation*}
\norm{Sf}\;=\;\int_{\RR}\int_{\RR_{>0}}\norm[\infty]{Sf(a,b)}\frac{da}{a^2}\,db%
\;=\;\int_{\RR}\int_{\RR}\norm[\infty]{f(s)(b)}\,ds\,db\;=\;\norm{f}.
\end{equation*}
With the notation $a'=e^t$ and $a=e^s$:
\begin{eqnarray*}
S(f\stackrel{\tilde{T}}{\ast} h)(a,b)&=&(f\stackrel{\tilde{T}}{\ast} h) (s)(-\frac{b}{a})%
\;=\;\int_{\RR}(\tilde{T}(t)f(s+t))\ast h(-t)(-\frac{b}{a})\,dt\\
&=&\int_{\RR}\int_{\RR}(\tilde{T}(t)f(s+t))(\frac{-b}{a}+v)\cdot h(-t))(-v)\,dv\,dt\\
&=&\int_{\RR}\int_{\RR}e^{-t}f(s+t)(\omega(-t,-\frac{b}{a}+v))\circ\omega(t,\cdot))\cdot \\
&&\quad h(-t))(-v)\,dv\,dt\\%
&=&\int_{\RR}\int_{\RR}e^{-t}f(s+t)(\omega(-t,-\frac{b}{a}-v))\circ\omega(t,\cdot))\cdot\\
&&\quad h(-t))(v)\,dv\,dt\\%
&=&\int_{\RR}\int_{\RR}e^{-t} f(s+t)(-e^{-t} (\frac{b}{a}+v)))\circ\omega(t,\cdot))\cdot\\
&&\quad h(-t))(v)\,dv\,dt\\%
&=&\int_{\RR}\int_{\RR}e^{-t} Sf(e^{s+t},e^{s}(\frac{b}{a}+v))%
\circ\omega(t,\cdot)\cdot\\%
&&\quad Sh(e^{-t},-e^{-t}v)\,dv\,dt\\
\end{eqnarray*}
\begin{eqnarray*}
&=&\int_{\RR}\int_{\RR_{>0}} Sf(aa',a(\frac{b}{a}+v))%
\circ\omega(\log(a'),\cdot)\cdot%
Sh(\frac{1}{a'},-\frac{v}{a'})\frac{da'}{a'^2}\,dv\\
&=&\int_{\RR}\int_{\RR_{>0}}\!\!\! Sf((a,b)(a',v))\circ\omega(\log(a'),\cdot)
Sh((a',v)^{-1})\,\frac{da'}{a'^2}\,dv\\
&=&\int_{\RR}\int_{\RR_{>0}}T'((a',v)Sf((a,b)(a',v))Sh((a',v)^{-1})\,\frac{da'}{a'^2}\,dv\\
&=& (Sf\stackrel{T'}{\ast} Sh)(a,b),
\end{eqnarray*}
For the $\ast$ operation on the respective algebras we have
\begin{eqnarray*}
(Sf)^{\ast}((a,b))&=& \Delta_{\AF}((a,b)^{-1})T'((a,b)^{-1})\overline{Sf((a,b)^{-1})},
\end{eqnarray*}
and
\begin{eqnarray*}
f^{\ast}(s)(b)&=&\tilde{T}(-s)\overline{f(-s)(-b)}
\end{eqnarray*}
So that
\begin{eqnarray*}
f^{\ast}(s)(b)&=&\tilde{T}(-s)(\overline{f(-s)(-b)}\\
&=&e^{s}\overline{f(-s)(-b)}\circ\omega(-s,\cdot)\\
&=&a\overline{Sf(\frac{1}{a},-\frac{b}{a})}\circ %
\omega(\log(\frac{1}{a},\cdot)\\
&=&\Delta_{\AF}((a^{-1},-\frac{b}{a}))T'((a^{-1},-\frac{b}{a}))\overline{Sf((a^{-1},-\frac{b}{a}))}\\
&=&\Delta_{\AF}((a,b)^{-1})T'((a,b)^{-1})\overline{Sf((a,b)^{-1})}\\
&=&(Sf)^{\ast}((a,b)).
\end{eqnarray*}
\par
Now we can embed $C_0(\RR)\to \Ruc({\AF})$ by extending a function,
already defined on the dilations to a function on $\AF$
which is independent of  the translation coordinate of a group element.
Restricting the left translation to this subspace
is just the action we denoted $T'$. So we obtain an embedding of $L^1(\AF,C_0(\RR),T')$ into
$L^1(\AF,\Ruc(\AF),T)$.
\end{proof}
\begin{cor}
$CD_{reg}(\AF)$ is not inverse-closed in $B(L^2(\AF))$.
\end{cor}
\begin{proof}
By the above theorem there is a selfadjoint operator in $CD_{reg}(\AF)$,
with a non-real spectrum. So  some operator  in $CD_{reg}(\AF)+\CC\mbox{id}$
is invertible in $B(L^2(\AF))$ but not in $CD_{reg}(\AF)+\CC\mbox{id}$.
\end{proof}

\end{document}